\documentclass[11pt]{amsart}
\usepackage{geometry} 
\usepackage{pgfplots}
\usepackage{psfrag}
\usepackage{amsmath}
\usepackage{calc}
\usepackage{systeme}
\usepackage{amsfonts}
\usepackage{amsthm}
\usepackage{algorithm2e}
\usepackage[applemac]{inputenc}
\usepackage{bbold}
\usepackage[numbers]{natbib}

\newcommand{\R}{\mathbb{R}}
\newcommand{\be}{\begin{equation}}
\newcommand{\ee}{\end{equation}}
\newcommand{\ba}{\begin{aligned}}
\newcommand{\ea}{\end{aligned}}
\newcommand{\eps}{\epsilon}

\newtheorem{Theorem}{Theorem}
\newtheorem{Lemma}{Lemma}
\newtheorem{Rem}{Remark}

\newtheorem{Prop}{Proposition}
\newtheorem{Cor}{Corollary}

\newcommand{\triplenorm}[1]{{\left\vert\kern-0.25ex\left\vert\kern-0.25ex\left\vert #1 
    \right\vert\kern-0.25ex\right\vert\kern-0.25ex\right\vert}}

\title{The equilibrium price of bubble assets}

\begin{document}





\author{Charles Bertucci \textsuperscript{a}, Jean-Michel Lasry\textsuperscript{b}, Pierre-Louis Lions\textsuperscript{b,c}
}
\address{ \textsuperscript{a} CMAP, CNRS, UMR 7641, \'Ecole polytechnique, IPP, 91120 Palaiseau, France\\ \textsuperscript{b} Universit\'e Paris-Dauphine, PSL Research University,UMR 7534, CEREMADE, 75016 Paris, France\\ \textsuperscript{c}Coll\`ege de France, 3 rue d'Ulm, 75005, Paris, France. }

\maketitle
\begin{abstract} Considering a simple economy, we derive a new Hamilton-Jacobi equation which is satisfied by the value of a "bubble" asset. We then show, by providing a rigorous mathematical analysis of this equation, that a unique non-zero stable solution exists under certain assumptions. The economic interpretation of this result is that, if the bubble asset can produce more stable returns than fiat money, agents protect themselves from hazardous situations through the bubble asset, thus forming a bubble's consensus value.
Our mathematical analysis uses different ideas coming from the study of semi-linear elliptic equations. 
\end{abstract}



\setcounter{tocdepth}{1}

\tableofcontents
\section{Introduction}
In his overlapping generations model of 1958, Samuelson \citep{samuelson1958} showed that, to protect themselves against the perishability of the natural goods, agents can form a consensus on the interest rate of a useless non-perishable good. With his simple model, conclusions have been drawn that a consensus on a money is a simple and efficient way to transfer capital to its future self. A main critic one could made to such a consequence is that, because it is controlled by governments, and thus the subject of hazardous public policies, the perishability of fiat money is questionable.

In this paper, we somehow reproduce the argument of Samuelson on money itself, by showing that to protect themselves against public policies, agents can form a consensus on useless goods, which do not need to be non-perishable, as long as they do not perish too much. By doing so, the agents create another money than fiat money. A typical example of such goods are cryptocurrencies, which present (at first) surprisingly stable prices, given that they do not yield returns. The existence of equilibrium prices of such bubble goods\footnote{By bubble good, or asset, we mean an asset which does not produce any return but still has a market value.} is now a classic topic of economics, stemming long before cryptocurrencies, see for instance Tirole \citep{tirole}, even if such equilibrium prices are quite often relying on the belief that the bubble will inflate. We shall see that in certain cases, the bubble can remain surprisingly stable. 

To explain the value of fiat money, Wallace \citep{wallace} showed that, despite being a priori a bubble, fiat money has a value because it facilitates exchanges on the market, thus proposing a fundamental upon which to value fiat money. 

Here we show that such other qualities are not necessary for a bubble good to have a stable value, as long as it almost non-perishable, since it can then serve as a reserve or assurance against future risks, thus placing ourselves much more in the spirit of Samuelson \citep{samuelson1958}. Note that the models mentioned above relied on overlapping generations, which we do not particularly need to consider here. We shall simply consider agents with a risk aversion and an economy with some risk on the public policy.\\

We shall end our paper, as Samuelson, with quite provocative conclusions. Namely we show that a bubble's consensus value (C value later on) exists and exhibits remarkable stability properties, which can be seen as counterintuitive for a bubble asset. Those conclusions stand upon a thorough mathematical analysis of a new equation which we derive as a consequence of a reached equilibrium on the price of the bubble asset. The main economic force at play in our model is the risk aversion of the agents, which incentives them to assure themselves against macro-economic risks, which they can do by forming a consensus on another good, provided it enjoys "sufficient" stability. This reasoning is quite obvious in itself and would not need any new mathematical development to be established. Nonetheless, our study puts this argument in a quantitative model that allows for further developments. For instance, our model exhibits a threshold on the "sufficient" stability of the bubble asset and since we can compute this threshold, we can give a precise answer to what "sufficient" means here.\\

We believe that attempts to derive, from rigorous mathematical analysis, stable equilibrium prices for bubble assets are quite scarce, and to the best of our knowledge, the work which bears some resemblance with this one is Biais, Rochet and Villeneuve \citep{brv}, upon which we comment more extensively later on in the paper.\\

The rest of our paper is organized as follows. In section \ref{sec:derivation}, we derive from a quite general model the equation to characterize equilibrium prices of bubble assets, which we analyze in details in section \ref{sec:math}. We then specify our model on the two examples of cryptocurrencies and real estate later on in section \ref{sec:eco}, then giving economic interpretations of our results, before concluding our work, namely with future directions of research. In particular, even if we advise to read this paper as a whole, economists might be tempted to jump directly from section \ref{sec:derivation} to section \ref{sec:eco}, while mathematicians could be more interested in section \ref{sec:math}.

\section{Generic model and derivation of the main equation}\label{sec:derivation}
We start by presenting our model in a simple one dimensional setting, before explaining how it can be easily generalized later on.
\subsection{The one dimensional case}
Time is continuous. There are two goods in the economy, one of them being the numeraire and the other one being called the "other" good. There is a quantity $K$ of "other" good. There is a finite number of agents whose decisions consists in choosing solely of the repartition of their capital between the two goods. We assume that the state of the world is described by an  element $x \in [0,1]$. At time $t$, the state of the world is denoted by $X_t\in [0,1]$. We assume that $(X_t)_{t \geq 0}$ is a Markov process on a standard probabilistic space. At time $t$, $X_t$ is observed by all agents and: is given as the strong solution of the stochastic differential equation
\[
dX_t = b(X_t)dt + \sigma dW_t,
\] 
which is assumed to be reflected at $0$ and $1$\footnote{We omit here the reflection term for the sake of clarity.}.
 When the world is in the state $x \in [0,1]$, we assume that the two goods produce (algebraic) interests given by $r_0(x)$ for the numeraire and $r_1(x)$ for the other good. The functions $r_0,r_1 : [0,1] \to \R$ are known by the agents. Note that we do not assume that the goods produce other types of return, hence the goods do not have intrinsic values.\\
 
Now that the physics of our model has been introduced, we can describe its economics. In the world we shall study, agents can exchange between them the two goods in a market. If the exchanges of the same good do not bear any interest, the exchanges of different goods yield the price of the second good expressed in the numeraire. We note $p_t$ this price at time $t$. We assume that exchanges are always possible at the market price $p_t$. Given this price, the agents can value their portfolio.  We assume that the agents have a utility function $U(\cdot)$, which only depends on the value, in numeraire, of their portfolio. Finally, because of the stationarity of our model, we assume that the agents seek to maximize the growth rate of their portfolio made of the two goods, valued in the numeraire.\\

An equilibrium price is a profile $(p_t)_{t \geq 0}$ such that for all time $t \geq 0$, the agents will agree to trade the goods at the price $p_t$. Such a price is indeed an equilibrium since, a priori, because the agents valued their portfolio in the numeraire, they are only ready to buy the "other" good, given that they anticipate that they will be able to sell it later on, as nobody wants to be stuck with unproductive valueless items.

 Our main objective is to characterize equilibrium prices in this market. In order to do so, we shall make the important assumption that the price is a function of the state of the economy, that is $p_t = u(X_t)$ for some function $u : [0,1] \to \R$. 
 \begin{Rem}
 We insist upon this assumption, as it is by no means obvious. In standard economics literature, quite different equilibria have been considered, namely bubble ones. Indeed if all agents share the common belief that $p_t$ is going to grow in an arbitrary large fashion, then it seems standard to prove that they can reach a consensus on such growth, thus proving existence of equilibria which are not of the form $p_t = u(X_t)$ but rather $p_t = v(t,X_t)$. We do not claim that such equilibria do not exist, nor that their study bears no interest, but our aim is here to show that a more intrinsic value of $p_t$ exists as we shall see, and by more intrinsic we mean that it is only a function of the state of the economy and not of some exterior belief.
 \end{Rem}
 
 \subsection{The decision problem of the agents}

 We now turn to the rather classical and not so difficult problem of the agents. Since we assumed they are eternal, it is somehow forced upon us to assume that they are going to maximize the growth rate of their wealth. Consider an agent with a wealth of $q$, of course expressed in the numeraire, at time $t$. The only decision of the agent is to decide how much it wants to invest in the other good. Observe that it is true because, since there are no frictions on the market, before making this investment decision, the agent can always sell all of what it owns in the "other" good, with the possibility of re-buying it instantly at the same price. 
 \begin{Prop}
 If all the functions involved are smooth, the problem of maximization of the growth rate of the portfolio of an agent with wealth $q$, when the economy is in state $x \in (0,1)$, is given by
\[
\ba
\max_{\theta\in \R} \quad &\bigg\{ qU'(q)\bigg( r_0(x)(1 - \theta) + r_1(x)\theta + \theta (u(x))^{-1} ( u'(x)\cdot b(x) + \frac{\sigma^2}{2} u''(x))\\
 &\quad \quad \quad\quad\quad +  \frac{U''(q)q}{U'(q)}\frac{\theta^2\sigma^2}{2u^2(x)}|u'(x)|^2\bigg)\bigg\},
\ea
\]
where the maximization is done over $\theta$, the proportion of "other" good the agent needs to buy.
 \end{Prop}
 This result is a consequence of the fact that, because of the assumption $p_t = u(X_t)$, we can express the future price $p_{t + dt}$ thanks to Ito's Lemma, as $p_{t +dt} = p_t + dt\{u'(X_t)b(X_t) +\sigma^2/2u''(X_t)\} + \sigma^2/2|u'(X_t)|^2dW_t$, as well as the fact that the numeraire yields $r_0(X_t)$ and the "other" good $r_1(X_t)$. The precise proof follows.
 \begin{proof}
We shall omit the terms arising from the reflection of the process, as we are able to do so if we consider everything before the first reaching time $\{0,1\}$, which is an adapted stopping time. If our agent decides to invest a proportion $\theta$ of its wealth at time $t$ in the "other" good, then at time $s \geq t$, its wealth $Y_{t,s}$ is given by
 
 \[
 Y_{t,s} = q(1 -\theta)e^{\int_t^{s}r_0(X_{t'})dt'} + \frac{\theta q}{p_t} p_{s}e^{\int_t^{s}r_1(X_{t'})dt'}.
 \]
For $s \in [t,t+h]$, we have
 \[
 dY_{t,s} = r_0(X_s)q(1 -\theta)e^{\int_t^{s}r_0(X_{t'})dt'} ds + \frac{\theta q}{p_t}  e^{\int_t^{s}r_1(X_{t'})dt'}dp_s + r_1(X_s)\frac{\theta q}{p_t}  p_{s}e^{\int_t^{s}r_1(X_{t'})dt'}ds.
 \]
 Using Ito's Lemma, we obviously have
 \[
 dp_s = u'(X_s)b(X_s)ds + \frac{u''(X_s)\sigma^2}{2}ds + u'(X_s)\sigma dW_s.
 \]
 Hence we arrive at
  \[
  \ba
 dY_{t,s} = \bigg[ &r_0(X_s)q(1 -\theta)e^{\int_t^{s}r_0(X_{t'})dt'}  + \frac{\theta q}{p_t}  e^{\int_t^{s}r_1(X_{t'})dt'}\left( u'(X_s)b(X_s)ds + \frac{u''(X_s)\sigma^2}{2}\right)\\
 & + r_1(X_s)\frac{\theta q}{p_t}  p_{s}e^{\int_t^{s}r_1(X_{t'})dt'}\bigg]ds + \frac{\theta q}{p_t}  e^{\int_t^{s}r_1(X_{t'})dt'}u'(X_s)\sigma dW_s.
 \ea
 \]
  Thus, we can write down, once again thanks to Ito's Lemma,
 \[
 \ba
U\left(Y_{t,t+h} \right) = U(q) + \int_t^{t+h}U'\left(Y_{t,s}\right)dY_{t,s} + \frac12\int_t^{t+h} U''(Y_{t,s})\frac{\theta^2q^2}{p_t^2}e^{2\int_t^{s}r_1(X_{t'})dt'}\sigma^2|\nabla_x u(X_s)|^2 ds.
 \ea
 \]
 Taking the conditional expectation on $X_t$, we arrive at the relation 
 \[
 \ba
\mathbb{E}[U\left(Y_{t,t+h} \right)|X_t] = U(q) + \mathbb{E}\bigg[&\int_t^{t+h}U'\left(Y_{t,s}\right)q\bigg( r_0(X_s)(1 -\theta)e^{\int_t^{s}r_0(X_{t'})dt'} + r_1(X_s)\theta\frac{ u(X_s)}{u(X_t)}e^{\int_t^{s}r_1(X_{t'})dt'}+\\
& + \frac{\theta }{u(X_t)}e^{\int_t^{s}r_1(X_{t'})dt'}(u'(X_s)\cdot b(X_s) + \frac{\sigma^2}{2} u''(X_s))\\
&+  \frac{U''(Y_{t,s})q}{U'(Y_{t,s})}\frac{\theta ^2\sigma^2}{2u^2(X_t)}e^{2\int_t^{s}r_1(X_{t'})dt'}| u'(X_s)|^2 \bigg)ds|X_t\bigg].
 \ea
 \]
The agent then should choose $\theta$ to maximize its earning's growth rate $\lim_{h \to 0}\frac{\mathbb{E}[U\left(Y_{t,t+h} \right)|X_t]- U(q) }{h}$, that is $\theta$ should be taken as the solution of
\[
\ba
\max_{\theta\in \R} \quad & qU'(q)\bigg( r_0(X_t)(1 - \theta) + r_1(X_t)\theta + \theta (u(X_t))^{-1} ( u'(X_t)\cdot b(X_t) + \frac{\sigma^2}{2} u''(X_t))\\
 &\quad \quad \quad\quad\quad +  \frac{U''(q)q}{U'(q)}\frac{\theta^2\sigma^2}{2u^2(X_t)}|u'(X_t)|^2\bigg).
\ea
\]
\end{proof}
The optimal level of investment $\theta$ always depends on the state of the economy $x$, and in general on $q$ as well. However, as usual for this type of problems, for some well known utility functions we can find explicit functions of $q$. For instance, for the CRRA utility function with parameter $\gamma \in (0,1)$, i.e. $U(q) = q^{1-\gamma}$, we find that the optimal level of investment does not depend on $q$ and is given by
\be\label{CRRA}
\theta^*(x) := u(x)\frac{ u(x)(r_1(x)- r_0(x)) +u'(x)\cdot b(x) + \frac{\sigma^2}{2} u''(x)}{\gamma\sigma^2|u' (x)|^2}.
\ee
\begin{Rem}
Note that the previous fraction is really of the form of the solution of the famous Merton's portfolio problem \citep{merton}. Indeed, the first term of the numerator represents the return of the "other" good relatively to the numeraire, the second term is simply the infinitesimal variation of price while the denominator is the squared of the volatility of the price $p_t$ multiplied by the risk parameter $\gamma$.
\end{Rem}
For the CARA utility function, $U(q) = 1- e^{-c q}$ for $c >0$, we find that the optimal $\theta$ is given by
\be\label{CARA}
\theta^*(q,x) = u(x) \frac{ u(x)(r_1(x)- r_0(x)) +u'(x)\cdot b(x) + \frac{\sigma^2}{2} u''(x)}{qc\sigma^2|u' (x)|^2}.
\ee
It now suffices to observe that in this case, the total demand, expressed in quantity of "other" good is constant and does not depend on the wealth of the agent. It is given by the quantity
\be\label{defDt}
\frac{ u(x)(r_1(x)- r_0(x)) +u'(x)\cdot b(x) + \frac{\sigma^2}{2} u''(x)}{c\sigma^2|u' (x)|^2}.
\ee

\subsection{Market clearing}
To close our model and finally derive the PDE satisfied by the price function $u$, we need to specify how does the market clear. On the side of the supply, the situation is quite simple, since we assume that the available quantity of "other" good is constant in time, given by $K$. Hence, we shall assume that any time, the supply is given by $K$.
\begin{Rem}
The fact that the supply $K$ is constant might seem surprising since we made the assumption that $r_1 \ne 0$. In cases where $r_1$ might be positive, this assumption is hard to verify, but in cases where $r_1$ is non-positive, it simply means that at every period of time, an amount of "other" good is added in the economy to match what has been lost through $r_1$.
\end{Rem}
On the side of the demand, the question is much more subtle and it can be expected since of course we are modelling prices of bubble assets!\\

\textbf{In the case of the CARA utility function}, we arrived at the conclusion that the total demand of all agents was given by $ND_t$ where $D_t$ is defined in \eqref{defDt}. Hence, in this situation, equaling supply and demand simply yields that
\[
u(X_t)(r_1(X_t)- r_0(X_t)) +u'(X_t)\cdot b(X_t) + \frac{\sigma^2}{2} u''(X_t) = \frac{cK\sigma^2}{N}|u'(X_t)|^2.
\]
Since this relation has to be satisfied for all $t$ and that the law of $X_t$ has full support on $(0,1)$, we arrive at the PDE
\be\label{pdev1}
u(x)(r_1(x)- r_0(x)) +u'(x)\cdot b(x) + \frac{\sigma^2}{2} u''(x) = \frac{cK\sigma^2}{N}|u'(x)|^2 \text{ in } (0,1).
\ee
Since, the process is reflected at $0$ and $1$, the previous PDE is naturally associated with Neumann boundary conditions $u'(x) = 0$ for $x \in \{0;1\}$.\\

\textbf{In the case of the CRRA utility funtion}, since we derived an optimal fraction of the wealth to be invested, independently of the individual wealth, we naturally obtain the market clearing condition
\[
 \left\{u(X_t)(r_1(X_t)- r_0(X_t)) +u'(X_t)\cdot b(X_t) + \frac{\sigma^2}{2} u''(X_t)\right\}Q_t = \gamma K\sigma^2|u' (X_t)|^2,
\]
where $Q_t$ is the total wealth of the population. The dependence of $Q_t$ on time prevents us to derive a proper PDE as in the previous case. The way we formulated our model, this dependence cannot be avoided as the total wealth is supposed to evolve according to evolution of $X_t$ through the dependencies upon it of $r_0,r_1$ and $u$. Nonetheless, under the assumption that, for some arbitrary exterior reasons, the total wealth of the population which is possibly invested in the "other" good is fixed at the level $Q$, we can indeed arrive at the PDE
\be\label{pdev2}
u(x)(r_1(x)- r_0(x)) +u'(x)\cdot b(x) + \frac{\sigma^2}{2} u''(x) = \frac{\gamma K\sigma^2}{Q}|u'(x)|^2 \text{ in } (0,1),
\ee
also associated with homogeneous Neumann boundary conditions. 

We remark that the form of this PDE is exactly the same of the previous one, which justifies the mathematical analysis of such equations in the next section. We also insist upon the fact that our assumption that $Q_t$ is constant is by no means trivial nor free of consequences. Nonetheless, we believe it gives an interesting idea of what can happen for the price of such assets, and that it is a plausible approximation of the reality. Note that a more sensible choice could have been to consider that\footnote{The notation $[\cdot]$ stands for the fact that $Q$ can depend on the whole function $u$, i.e. in a non-local manner.} $Q_t = Q[u](X_t)$. We are then able to also obtain a PDE which is given by
\be\label{pdev3}
u(x)(r_1(x)- r_0(x)) +u'(x)\cdot b(x) + \frac{\sigma^2}{2} u''(x) = \frac{\gamma K\sigma^2}{Q[u](x)}|u'(x)|^2 \text{ in } (0,1),
\ee


Note that $u \equiv 0$ is always a solution of \eqref{pdev1}, \eqref{pdev2} or \eqref{pdev3}, which is natural since, if the anticipated price of the good is $0$, no one is interested in it and its current price is then also $0$. In the next section, we provide a detailed mathematical analysis of such equations. We shall prove that, under certain assumptions, a non-zero solution of this PDE exists, and enjoys remarkable stability properties. We then come back on economical implications of such a fact in section \ref{sec:eco}.

\subsection{Generalization of our model}
Before entering the mathematical analysis of the next section, we indicate a simple and immediate extension of our model. Assume that the process $(X_t)_{t \geq 0}$ does not live in $(0,1)$ but rather in some smooth bounded domain $\Omega \subset \R^d$, that it is still reflected at the boundary and a solution of the SDE
\[
dX_t = b(X_t)dt + \sigma(x) dW_t,
\]
for $(W_t)_{t \geq 0}$ a standard $d$ dimensional Brownian motion. In addition, assume that the "other" good produces, in a time interval $[t, t+h]$, a revenue expressed in numeraire equal to $\int_{t}^{t+h}f(X_s)ds$, with $f \geq 0$ a smooth given function. Note that $f$ can be equal to $0$ in which case we fall back into the previous situation. In such a situation, the PDE at which we arrive is given by
\be\label{pdev4}
u(x)(r_1(x)- r_0(x)) +\nabla_xu (x)\cdot b(x) + \frac{\sigma^2(x)}{2} \Delta u(x) + f(x) = C|\nabla_x u(x)|^2 \text{ in } \Omega,
\ee
once again with homogeneous Neumann boundary conditions, and where $C$ is equal to $\frac{cK\sigma^2}{N}$ in the case of CARA utility, to $\frac{\gamma K\sigma^2}{Q}$ in the case of the CRRA utility with constant total wealth, and to $\frac{\gamma K\sigma^2}{Q[u]}$ in the case of the CRRA utility with the assumption that $Q_t = Q[u](X_t)$.

\begin{Rem}
Note that the previous modelling is quite in the spirit of the mean field games theory developed by the last two authors \citep{lasry2007mean}. Indeed, even though no proper game is introduced, we derived an equation satisfied by a macro-economic quantity (the price function $u$) from the aggregation of small agents. Remark that we did not insisted upon it, but we indeed assumed that agents were small in the sense that they are price takers, and do not consider that they are going to affect the future price with their behaviour, which amounts to say that they only have an infinitesimal effect on macroscopic quantities.
\end{Rem}

\section{Mathematical analysis}\label{sec:math}
In this section, we analyze the equation
\be\label{kr}
\ba
-\nu\Delta u + \frac 12 |\nabla_x u|^2 = a(x)u \text{ in } \Omega,\\
u(x) > 0 \text{ in } \Omega,\\
\partial_n u = 0 \text{ on } \partial \Omega.
\ea
\ee
where $\Omega$ is a smooth bounded domain of $\R^d$, $\nu > 0$ and $a$ is a smooth function on which additional assumptions shall be made later on.\\

 Equation \eqref{kr} is the general equation derived in the previous section, except for the presence of the drift term $b$ that we omit here since it does not play a strong role in the following and adapting our results to the presence of $b$ is trivial. Equation \eqref{kr} is a quadratic stationary Hamilton-Jacobi-Bellman equation. Compared to the existing mathematical literature, the main novelty of \eqref{kr} is the fact that $a(\cdot)$ depends on $x$ together with the quadratic non-linearity and the stationarity. If one of the previous three elements was absent, then the problem will be classical, whereas here it exhibits new properties.
 
 As we shall see, \eqref{kr} can be interpreted as a sort of non-linear version of Krein-Rutman Theorem. This link will only be made explicit at the end of this section, but some ideas, namely ones introduced by P.-L. Lions in \citep{lions2021cours} for Krein-Rutman like results, shall be used during our analysis.\\

When it is well defined, we shall denote by $\lambda_1(A)$ the first (smallest) eigenvalue of the differential operator $A$ on $\Omega$ with Neumann boundary conditions. Furthermore, we shall use the notation $u >> 0$ to indicate that the function $u: \Omega \to \R$ is bounded from below by a positive constant on $\Omega$. Finally note that for operators $A$ of the form $-\nu \Delta - a $ for some smooth function $a$, $\lambda_1(A)$ is well defined and there is an associate eigenfunction $\phi >>0$.

\subsection{Existence and uniqueness}
This section establishes the main result of the paper, which is the existence and uniqueness of a solution to \eqref{kr}.
\begin{Theorem}\label{thm:1}
Assume that \begin{itemize} \item $\lambda_1(-\nu\Delta - a)< 0$,\item $\min_\Omega a < 0$. \end{itemize} Then, there exists a unique smooth solution of \eqref{kr}.
\end{Theorem}
\begin{Rem}
Note that that the first assumption implies that $\max_\Omega a > 0$.
\end{Rem}
We separate the proof of Theorem \ref{thm:1} into several lemmas. The first one is concerned with the uniqueness part of the result and shall be crucial in the remaining of the paper.
\begin{Lemma}\label{lemma:uniq}
Assume that the function $a$ takes both positive and negative values. Consider $u>>0$ and $v>>0$ two smooth functions on $\Omega$. Assume that 
\[
\ba
-\nu\Delta u + \frac 12 |\nabla_x u|^2 \leq a(x)u \text{ in } \Omega,\\
-\nu\Delta v + \frac 12 |\nabla_x v|^2 \geq a(x)v \text{ in } \Omega.
\ea
\]
Then, $u \leq v$ and either $u = v$ or $u < v$.
\end{Lemma}
\begin{proof}
 The argument follows a technique from Laetsch \citep{laetsch1975uniqueness}. Since $\Omega$ is compact, we can consider $\theta \in (0,1]$ defined by 
 \[
 \theta = \max \{ \tilde \theta \in [0,1], \forall x \in \Omega, \tilde\theta u(x) \leq v(x)\}.
 \]
Observe that $u_\theta := \theta u$ is a solution of
\[
-\nu\Delta u_\theta + \frac {1}{2} |\nabla_x u_\theta|^2 + \left(\frac{1}{2\theta} - \frac 12\right) |\nabla_x u_\theta|^2\leq a(x)u_\theta \text{ in } \Omega.
\]
Hence, $w = v - u_\theta$ is a solution of
\[
-\nu\Delta w + \frac12\left( |\nabla_x v|^2 - |\nabla_x u_\theta|^2\right) \geq a(x)w \text{ in } \Omega.
\]
Define the vector field $B$ with
\[
B(x) = \frac{\left( |\nabla_x v|^2 - |\nabla_x u_\theta|^2\right)}{2(\nabla_x v - \nabla_x u_\theta)}.
\]
Clearly, $B$ is well defined and smooth. Observe now that both $w \geq 0$ and 
\be\label{eq1}
-\nu\Delta w + B\cdot \nabla_x w \geq a(x)w \text{ in } \Omega.
\ee
This implies that, for $\lambda \geq 0$
\be
-\nu\Delta w + B\cdot \nabla_x w + (\lambda - a(x))w\geq \lambda w \geq 0 \text{ in } \Omega.
\ee
Assume that $w \ne 0$. Then, choosing $\lambda > 0$ sufficiently large, by strong maximum principle, $\forall x \in \Omega, w(x) > 0$, which gives $\theta = 1$.

If $w = 0$, then $u_\theta$ solves the same inequality as $v$, since they are equal. This implies that $\left(\frac{1}{2\theta} - \frac 12\right) |\nabla_x u_\theta|^2= 0$ which either implies that $\theta = 1$ or that $u$ is constant, which is impossible by construction. Thus $\theta = 1$ and the result is proved. Note that if $u \ne v$, then we already saw that $\theta =1$ and $\forall x \in \Omega, w(x) > 0$. \end{proof}

We now turn to an a priori estimate which shall constitute a compactness argument in the existence part of Theorem \ref{thm:1}.
\begin{Lemma}\label{lemma:apriori}
Assume that the function $a$ takes both positive and negative values. Then, there exists a constant $C >0$ depending only on the function $a$ and $\Omega$ such that, for any smooth non-negative function $u$ satisfying
\[
\ba
-\nu \Delta u + \frac12|\nabla_x u|^2 \leq a(x) u \text{ in } \Omega,\\
\partial_n u = 0 \text{ on } \partial \Omega,
\ea
\]
we have $u \leq C$.
\end{Lemma}

\begin{proof}
 Integrating the inequality and using the Neumann boundary condition leads to 
\be\label{ineq}
\int_\Omega |\nabla u|^2 \leq 2 \max_\Omega a \int_\Omega u.
\ee
We now claim that the previous, together with the inequality satisfied by $u$, imply a bound on $\|u\|_\infty$ which depends only $\Omega$ and $a$. We first a show a bound in $L^1(\Omega)$. Assume that there exists a sequence $(u_n)_{n \geq 0}$ as in the Lemma with $\| u_n\|_{L^1} \to \infty$ as $n \to \infty$. Observe that for all $n$, $w_n := u_n/\| u_n\|_{L^1}$ satisfies
\[
-\nu\Delta w_n + \| u_n\|_{L^1}\frac12 |\nabla_x w_n|^2 \leq a(x)w_n \text{ in } \Omega.
\]
Integrating this relation we obtain that $\int_\Omega |\nabla_x w_n|^2\to 0$ as $n \to \infty$. Furthermore, we know that $(w_n)_{n \geq 0}$ is bounded in $L^{1}(\Omega)$. Hence, it converges up to a subsequence toward a positive constant denoted by $c$, in $W^{1,1}(\Omega)$ for instance. From the normalization in $L^1(\Omega)$, this constant is $c =(\int_\Omega)^{-1}$. Hence the full sequence converges toward $c$. Thus, passing to the limit in the previous equation in the sense of distributions implies that the product $aw \geq 0$, which is a contradiction because we assumed that $a$ is not of constant sign.

Hence, the set of such $u$ is bounded in $L^1(\Omega)$, thus, thanks to \eqref{ineq}, it is bounded in $H^1(\Omega)$ since $\|u\|_{L^1} + \|\nabla u\|_{L^2}$ is an equivalent norm in $H^1(\Omega)$. 

Take $K > \max a$ and for any $u$ as in the Lemma, consider $v$ the solution of 
\[
\ba
-\nu \Delta v + K v = a(x) u + K u \text{ in } \Omega,\\
\partial_n v = 0 \text{ in } \partial \Omega.
\ea
\]
Since $u$ is bounded in $H^1(\Omega)$ by a constant $C(\nu,a)$, we deduce that $v$ is bounded in $H^3(\Omega)$ by a constant $C(\nu,a,K)$. Since $u$ is a subsolution of the PDE satisfied by $v$, we deduce that $u \leq v$. If $d < 6$, $v \in L^\infty(\Omega)$, if $d > 6$, $v \in L^{\frac{2n}{n-6}}$ and $v \in L^p(\Omega)$ for any $p< \infty$ if $d = 6$, with the fact that each time, the $L^p$ norm of $v$ is bounded by a constant $C(\nu,a,K)$. Hence, if $d < 6$ the result is proved. We have established that $u$ is bounded in say $L^{\frac{2n}{n-5}}(\Omega)$ by some constant $C(\nu,a,K)$ if $d \geq 6$. Repeating the same argument then yields the required result.

\end{proof}

We now provide the existence of (strictly) positive sub-solution of \eqref{kr}
\begin{Lemma}\label{lemma:sub}
Assume that $\lambda_1(-\nu \Delta - a)< 0$, then there exists a smooth $u_0>>0$ such that 
\[
\ba
-\nu\Delta u_0 + \frac 12 |\nabla_x u_0|^2 \leq a(x) u_0 \text{ in } \Omega,\\
\partial_n u_0 = 0 \text{ on } \partial \Omega.
\ea
\]
\end{Lemma}
\begin{proof}
Consider $-\lambda := \lambda_1(-\nu \Delta - a)$ and $u_0$ an associated positive eigenvector. Clearly, since $u_0$ is smooth, choosing $\alpha u_0$ instead of $u_0$ for $\alpha> 0$ sufficiently small yields that $\frac12|\nabla_xu_0|^2 \leq \lambda u_0$. Hence, since $-\nu \Delta u_0 + \lambda u_0 = a(x) u_0$, the result follows.
\end{proof}

We can now state the proof of our main result.
\begin{proof}[Proof of Theorem \ref{thm:1}]
The uniqueness part is an immediate consequence of Lemma \ref{lemma:uniq}. Take $u_0$ given by Lemma \ref{lemma:sub}, a constant $K> 0$ such that $K > \max_\Omega a$ and consider now the sequence $(u_n)_{n \geq 0}$ defined by 
\[
\ba
-\nu\Delta u_{n+1} + \frac12|\nabla_x u_{n+1}|^2 + (K-a(x))u_{n+1} = Ku_n \text{ in } \Omega,\\
\partial_n u_{n+1} = 0 \text{ on } \partial \Omega.
\ea
\]
This sequence is well defined. We now prove that it is non-decreasing. Indeed, remark that $-\nu \Delta u_1 + \frac 12 |\nabla_x u_1|^2 + (K-a(x))u_1 = Ku_0$ and that $-\nu \Delta u_0 + \frac 12 |\nabla_x u_0|^2 + (K-a(x))u_0 \leq Ku_0$. Hence, by standard comparison principle, we obtain that $u_1 \geq u_0$. This implies that $u_1$ is also a sub-solution of the equation. Hence, arguing by induction, we obtain $(u_n)$ is a non-decreasing sequence of sub-solutions of the problem. Hence, from Lemma \ref{lemma:apriori}, we deduce that the sequence is bounded. Hence, it converges toward some limit $u$. By standard arguments of regularity, we obtain that $(u_n)$ is bounded in some functional space with high regularity, and thus $u$ is indeed a solution of the problem, namely by passing in the limit in the PDE satisfied by $u_{n+1}$.
\end{proof}
\begin{Rem}
The previous proof yields a natural numerical scheme to compute the solution of \eqref{kr}.
\end{Rem}

%
%

\subsection{On the necessity of the assumptions of the main theorem}
We highlight the fact that the assumptions of Theorem \ref{thm:1} are also necessary in some sense. Indeed, we have the following result.
\begin{Prop}
Consider a smooth solution $u \geq 0$ of 
\[
\ba
-\nu \Delta u + \frac 12 |\nabla_x u|^2 = a(x) u \text{ in } \Omega,\\
\partial_n u = 0 \text{ in } \partial \Omega.
\ea
\]
If $a \ne 0$ and $\lambda_1(-\nu\Delta -a ) \geq 0$, then $u = 0$.
\end{Prop}
\begin{proof}
Denote $\lambda := \lambda_1(-\nu\Delta -a ) \geq 0$, and consider $\phi >> 0$ an eigenfunction of $-\nu \Delta - a$ associated to $\lambda$. Define $\alpha = \min \{ \beta \geq 0 | u \leq \beta \phi\}$ and $w = u - \alpha \phi$. By construction, $w \leq 0$. Furthermore,
\[
-\nu \Delta w  - \lambda w = a(x) w -\frac 12 |\nabla_xu|^2 -\lambda u \text{ in } \Omega.
\]
Hence, for any $K \geq 0$, we obtain that $w \leq 0$ and
\[
-\nu \Delta w - \lambda w -a(x)w + Kw \leq K w \leq 0 \text{ in } \Omega.
\]
Hence, by strong maximum principle, we obtain that $w = 0$, because $-w >>0$ is not possible since by construction, $w(x) = 0$ has a solution $x \in \Omega$. Hence $u = \alpha \phi$. Hence $u$ is also an eigenfunction of $-\nu \Delta -a$ and thus we deduce that $\frac12|\nabla_x u|^2 = -\lambda u$. Thus, $u$ is constant, but since $a \ne 0$, we find that $u = 0$. 
\end{proof}
\begin{Prop}
If $a \geq 0$ and there exists a solution $u$ to \eqref{kr}, then $a \equiv 0$.
\end{Prop} 
\begin{proof}
In this case, we have $$-\nu \Delta u + \frac 12 |\nabla u|^2 \geq 0$$ which implies that $u$ is constant. Hence, the only solution is $a =0$.
\end{proof}

\subsection{The case of a source term}
We here explain how to adapt the results of the previous section to the case of
\be\label{krf}
\ba
-\nu\Delta u + \frac 12 |\nabla_x u|^2 = a(x)u + f \text{ in } \Omega,\\
u>> 0,\\
\partial_n u = 0 \text{ on } \partial \Omega,
\ea
\ee
with the same assumption as above, but now $f \geq 0$ is a smooth function. We have the following result.
\begin{Theorem}\label{thm:2}
Under the assumptions of Theorem \ref{thm:1}, there exists a unique smooth solution of \eqref{krf}.
\end{Theorem}
\begin{proof}
We only insist briefly on the main differences with the previous case. The existence part is mostly left unchanged, as the existence also follows on the same type of a priori estimate, except that now \eqref{ineq} is replaced by
\[
\int_\Omega |\nabla u|^2 \leq 2 \max_\Omega a \int_\Omega u + \int_\Omega f.
\]
On the other hand, uniqueness is obtained by adapting straightforwardly Lemma \ref{lemma:uniq} and realizing that if $u$ is a smooth solution of
\[
-\nu \Delta + \frac12|\nabla_x u|^2 \leq a(x) u + f \text{ in } \Omega,
\]
then $u_\theta = \theta u$ is also a solution of the same inequality because $\theta f \leq f$ because $f \geq 0$.
\end{proof}
As an immediate consequence of Lemma \ref{lemma:comp}, we obtain also the following.
\begin{Cor}
If we denote by $u(f)$ the solution of \eqref{krf} for a smooth $f \geq 0$, then $u$ is increasing with respect to $f$.
\end{Cor}
\begin{Rem}
We insist upon the fact that the assumption $f\geq 0$ is here crucial to extend Lemma \ref{lemma:uniq}. 
\end{Rem}

\subsection{Optimal control interpretation}\label{sec:optcontr}

Consider an atomeless filtered probabilistic space $(\Omega, \mathcal{A},(\mathcal{F}_t)_{t\geq 0}, \mathbb{P})$ and a Brownian motion $(W_t)_{t \geq 0}$ upon it. We want to give a stochastic optimal control interpretation of the solution of \eqref{kr}. Of course, the natural associated stochastic optimal control problem is given, for an initial condition $x \in \Omega$ by 
\be\label{opt}
\inf_{(\alpha_t)_{t \geq 0}}\frac12\mathbb{E}\left[ \int_0^{+\infty}e^{\int_0^ta(X^{x,\alpha}_s)ds}|\alpha_t|^2dt\right],
\ee
where the infimum is taken over square integrable admissible processes and the process $(X^{x,\alpha}_t)_{t \geq 0}$ is given by
\be\label{sde}
X^{x,\alpha}_t = x + \int_0^t\alpha_s ds + \sqrt{2\nu}W_t + k_t,
\ee
where $(k_t)_{ t \geq 0}$ is the standard process associated to the normal reflexion on the boundary of $\Omega$, see Lions and Sznitman \citep{lionssznitman}. Of course this problem is trivially solved by taking $\alpha_t = 0$ for all $t \geq 0$. This corresponds to the solution $0$ of the PDE in \eqref{kr}. To develop a more interesting analysis and recover the solution $u$ of \eqref{kr} given by Theorem \ref{thm:1}, we introduce a perturbed version of the previous problem. Namely, we take $\eps > 0$ and we are interested in 
\[
v_\eps (x) := \inf_{(\alpha_t)_{t \geq 0}}\mathbb{E}\left[ \int_0^{+\infty}e^{\int_0^ta(X^{x,\alpha}_s)ds}\left(\frac12|\alpha_t|^2 + \eps\right)dt\right],
\]
where $(X^{x,\alpha}_t )_{t \geq 0}$ is still given as the solution of \eqref{sde} and the infimum is still taken over admissible squared integrable process.
We want to establish the following result.
\begin{Prop}\label{prop:contr}
\[
\lim_{\eps \to 0^+}\|v_\eps -u\|_\infty = 0. 
\]
\end{Prop}
In order to do so, we start by recalling a standard Lemma.
\begin{Lemma}\label{lemma:41}
Let $x \in \Omega$ and $(Y_t)_{t \geq 0}$ be the solution of
\[
dY_t = b(Y_t) dt + \sqrt{2\nu}dW_t + dk_t,
\]
with initial condition $Y_0 = x$, where $(W_t)_{t \geq 0}$ is a standard $d$ dimensional Brownian motion and $(k_t)_{t \geq 0}$ the process ensuring the fact that $(Y_t)_{t \geq 0}$ is reflected on $\partial \Omega$ given by $|k|_t = \int_0^t\mathbb{1}_{Y_s \in \partial \Omega}d|k|_s$ and $k_t = \int_0^t n(Y_s)d|k|_s$ where $n(x)$ is the unit outward normal vector to $\Omega$ at $x$. Then,
\[
\mathbb{E}\left[e^{\int_0^Ta(Y_t)dt}\right] =_{T \to \infty} O\left(e^{-\lambda_1(-\nu \Delta -b \cdot\nabla -a)T}\right).
\]
\end{Lemma}
\begin{proof}
Consider $\phi>>0$ an eigenfunction associated to $\lambda_1(-\nu \Delta- b\cdot \nabla -a)$. From the PDE satisfied by $\phi$, we know from standard representation results that for any $T > 0$
\[
\phi(x) = \mathbb{E}\left[ e^{\int_0^Ta(Y_t)+ \lambda_1(-\nu\Delta - b\cdot \nabla - a) dt}\phi(Y_T)  \right].
\]
Since there exists $C > 0$ such that $C^{-1}\leq \phi \leq C$, the required result follows.
\end{proof}
\begin{Rem}
This Lemma gives another point of view on the consequence of the assumption $\lambda_1(-\nu\Delta -a) < 0$ in Theorem \ref{thm:1}.
\end{Rem}
Hence, under the assumptions of Theorem \ref{thm:1}, the exponential term in \eqref{opt} grows exponentially in time in expectation if the control $0$ is chosen. Thus, it has to be compensated in order for the integral to make sense, which prohibits the use of the control $0$, and thus justifies once again that the associated value cannot be $0$.

\begin{proof}[Proof of Proposition \ref{prop:contr}]
Consider $u_\eps$ the unique solution of
\be\label{eq3}
\ba
-\nu\Delta u_\eps + \frac12|\nabla_x u_\eps|^2 = a(x) u_\eps + \eps \text{ in } \Omega,\\
\partial_n u_\eps = 0 \text{ on } \partial \Omega,\\
u_\eps >0 \text{ in } \Omega,
\ea
\ee
given by Theorem \ref{thm:2}.
We want to establish that $v_\eps = u_\eps$. Choosing $\tilde \alpha_t = -\nabla_x u_\eps(X^{x}_t)$, we obtain, thanks to the previous Lemma that $\mathbb{E}[\int_0^\infty e^{\int_0^t a(X^{x,\tilde \alpha}_s)ds}dt]$ is bounded, hence since $\nabla_x u_\eps$ is bounded, we also deduce that $v_\eps < \infty$. Moreover, this also implies that $v_\eps \leq u_\eps$. Furthermore, we obviously have that $v_\eps \geq - \eps / (\min_\Omega a) > 0$. We can also obtain using classical techniques of optimal control that $v_\eps$ is Lipschitz continuous, see for instance Fleming and Soner \citep{fleming}, section IV.8\footnote{The results of \citep{fleming} are not directly usable in this context. Nonetheless the techniques of proof which are used in the section we mentioned are, as they rely on the, local in time, controllability of the system.}. In particular, we can now state that the dynamic programing principle holds, and it is now classical to establish that $v_\eps$ is a smooth solution of \eqref{eq3}, and by uniqueness of such solutions, we obtain that $u_\eps = v_\eps$. From Lemma \ref{lemma:uniq}, we know that $v_\eps \downarrow v$ uniformly to some function $v \geq u$ as $\eps \to 0$. From the equation and standard elliptic regularity estimates, we know that $v$ is smooth and satisfies \eqref{kr}, and thus, by uniqueness, $v = u$ and the result is proved.

\end{proof}
\begin{Rem}
Hence, if $u$ is obviously not the value function of \eqref{opt}, it can be approximated by value functions of "nearby" optimal control problems. In terms of control, the interpretation is quite clear. If playing $\alpha_t \equiv 0$ is of course optimal, it leads to the integral in front of it to be arbitrary large, namely because of the assumption $\lambda_1(-\nu \Delta - a) < 0$ as highlighted in the next result. Somehow, the control $\alpha_t := -\nabla_x u(X_t)$ does not allow to obtain a cost of $0$, but it does allow for $\mathbb{E}[\int_0^\infty e^{\int_0^\infty a(X_s)ds}dt] < \infty$, which itself allow for "errors" in the control, without having to pay an infinite cost.
\end{Rem}
We believe that this is a strong argument for the stability of the solution of \eqref{kr}.\\

\begin{Rem}
The optimal control interpretation of \eqref{kr} that we just presented could also give rise to another proof of uniqueness of solutions of \eqref{kr}, namely by showing that any solution is equal to the limit of the $v_\eps$.
\end{Rem}

\subsection{Stability of the solution}
We now indicate a stability argument of the unique solution of \eqref{kr} with respect to the term $a$.
\begin{Prop}\label{prop:stab}
Consider a sequence $(a_n)_{n \geq 0}$ of smooth real functions on $\Omega$, bounded in $C^1(\bar \Omega)$, such that for all $n \geq 0$, the assumptions of Theorem \ref{thm:1} are satisfied. Assume also that $(a_n)_{n\geq 0}$ converges uniformly on $\Omega$ toward some smooth function $a$, which is not of constant sign. Consider the sequence $(u_n)_{n \geq 0}$ of unique solutions of \eqref{kr} associated to $(a_n)_{n\geq 0}$. Then, $(u_n)_{n \geq 0}$ converges uniformly toward a smooth non-negative function $u$, and either $\lambda_1(-\nu\Delta - a) = 0$ and $u = 0$ or $\lambda_1(-\nu\Delta - a ) < 0$ and $u$ is the unique solution of \eqref{kr} associated to $a$.
\end{Prop}
\begin{proof}
From the proof of Lemma \ref{lemma:apriori}, we know that $(u_n)_{n \geq 0}$ is a bounded sequence of $C^2(\bar \Omega)$. Hence, it has a limit point $v$ which satisfies \eqref{kr}. Either $\lambda_1(-\nu\Delta - a ) = 0$ and thus $v = 0$, or $\lambda_1(-\nu\Delta - a ) < 0$ and from the proof of Lemma \ref{lemma:sub}, we can found a uniform in $n$ sub-solution $\phi$ of 
\[
\ba
-\nu\Delta u + \frac 12 |\nabla_x u|^2 \leq a_n(x) u \text{ in } \Omega,\\
\partial_n u = 0 \text{ on } \partial \Omega,
\ea
\]
such that $\phi$ is bounded from below by a positive constant. Hence, so is $v$. Thus by Lemma \ref{lemma:uniq}, we conclude that the result holds.
\end{proof}

The previous result justifies the phenomenon which is described in Figure $1$ below.

\subsection{Some properties of the solution}

\begin{Prop}\label{prop:mono}
Take $\lambda > 0$ and consider the unique solution $u_\lambda$ of 
\[
\ba
-\nu& \Delta u_\lambda +  \lambda \frac12|\nabla_x u_\lambda |^2 = a(x) u_\lambda \text{ in } \Omega,\\
&\partial_n u = 0 \text{ on } \partial \Omega,\\
&u_\lambda > 0 \text{ in } \Omega.
\ea
\]
Then the map $\lambda \to u_\lambda$ is decreasing with, for all $x \in \Omega$, $\lim_{\lambda \to 0} u_\lambda(x) = +\infty$ and $\lim_{\lambda \to 0} u_\lambda(x) = 0$.
\end{Prop}
\begin{proof}
The fact that the map is decreasing simply follows from Lemma \ref{lemma:uniq}. The limit can be observed easily by remarking that in this quadratic case, $\lambda u_\lambda = u_1$.
\end{proof}
\begin{Prop}\label{prop:mono2}
Consider $a_1 \leq a_2$ such that for $i = 1,2$, the assumptions of Theorem \ref{thm:1} are satisfied. We note $u_1$ and $u_2$ the respective solutions. It then holds that $u_1 \leq u_2$.
\end{Prop}
\begin{proof}
If this can be seen at the level of the optimal control interpretation, it is an immediate consequence of Lemma \ref{lemma:uniq} apply on $a_1$ since $(a_2 - a_1)u_2 \geq 0$.
\end{proof}

We now prove a monotone property of the solution in a particular case.
\begin{Prop}\label{prop:5}
Assume that $\Omega = (0,1)$, that the assumptions of Theorem \ref{thm:1} hold and that $a$ is a non-decreasing function, then so is $u$, the unique solution of \eqref{kr}.
\end{Prop}
\begin{proof}
Consider for $\eps > 0$, $u_\eps$ the unique solution of 
\[
\ba
-\nu\partial_{xx}u_\eps + \frac12|\partial_x u_\eps|^2 = a(x) u_\eps + \eps \text{ in } \Omega,\\
\partial_n u_\eps = 0 \text{ on } \Omega.
\ea
\]
Take $0 < x\leq y < 1$ and consider, in the optimal control formulation of the problem, an optimal control $\alpha_t$ for the problem starting in $y$. Note that we can choose $\alpha_t = -\partial_x u(t,X_t)$. The associated state evolves according to
\[
dX^{y}_t = \alpha_t dt + \sqrt{2\nu}dW_t + dk_t,
\]
where $(k_t)$ is defined as in Lemma \ref{lemma:41} and $(W_t)_{t \geq 0}$ is a standard Brownian motion. Let us now remark that we can also choose the same control $\alpha_t$ in the problem starting form $x$, that the associated state then evolves according to
\[
dX^{x}_t = \alpha_t dt + \sqrt{2\nu}dW_t + dk'_t,
\]
for some other process $(k'_t)_{t \geq 0}$ defined similarly. From standard properties of reflected diffusions \citep{lionssznitman}, we know that, almost surely $X^x_t \leq X^y_t$. Hence, since
\[
u_\eps(x) - u_\eps(y) \leq \mathbb{E}\left[ \int_0^\infty \left(e^{\int_0^ta(X^x_s)ds} -e^{\int_0^ta(X^y_s)ds}\right)\left(\frac12|\alpha_t|^2 + \eps\right)dt \right].
\]
The right hand side is non-positive since $a$ is non-decreasing. Hence $u_\eps$ is non-decreasing and the result follows by passing to the limit $\eps \to 0$.
\end{proof}

\subsection{Interpretation of Theorem \ref{thm:1} as a non-linear version of Krein-Rutman}
In this section we give a visual illustration of the type of phenomenon which is happening, as well as mathematical justifications for the illustration we provide. We shall consider here the problem
\be\label{pb2}
\ba
-\nu \Delta u + \eps|\nabla_x u|^2 + u = \lambda r(x) u \text{ in } \Omega,\\
u \geq 0 \text{ in } \Omega,\\
\partial_n u = 0 \text{ on } \partial \Omega,
\ea
\ee
where $r$ is a smooth non-negative non-constant function which vanishes at some point such that $\lambda_1(-\nu\Delta -(r-1) ) = -1$. In this section, $\eps$ and $\lambda$ are parameters that are aimed to vary. \\

From the assumptions on $r$, there exists $\lambda_c \in (0,1)$ such that for $\lambda \in [0,\lambda_c]$, for any $\eps > 0$, the unique solution of \eqref{pb2} is $0$. More precisely, $\lambda_c$ is such that $\lambda_1(-\nu\Delta - \lambda_c r +1) = 0$. Then, for $\lambda > \lambda_c$, for any $\eps >0$ there are two solutions to \eqref{pb2}, the constant $0$ and a solution $u_{\eps,\lambda}$ given by Theorem \ref{thm:1}. Furthermore, thanks to Proposition \ref{prop:stab}, we know that, for any $\eps > 0$, the function $\lambda \to u_{\eps,\lambda}$ is continuous on $(\lambda_c,1]$. This provides, for any $\eps > 0$, a continuum of functions indexed by $\lambda$, which goes out from $0$ at $\lambda_c$. (The continuity when $\lambda \to \lambda_c$ is rather obvious). Finally, for $\eps' < \eps$, we know that for any $\lambda > \lambda_c$, $u_{\eps',\lambda} \geq u_{\eps,\lambda}$ thanks to Proposition \ref{prop:mono}.\\

On the other hand, when $\eps = 0$ we know that solutions of \eqref{pb2} are only possible for $\lambda =\lambda_c$, locally around $\lambda_c$, and all those solutions are the same up to a multiplication by a positive constant. These phenomena are summarized in Figure \ref{fig:1}.\\

\begin{figure}
\centering
  \centering
  \includegraphics[width=0.6\linewidth]{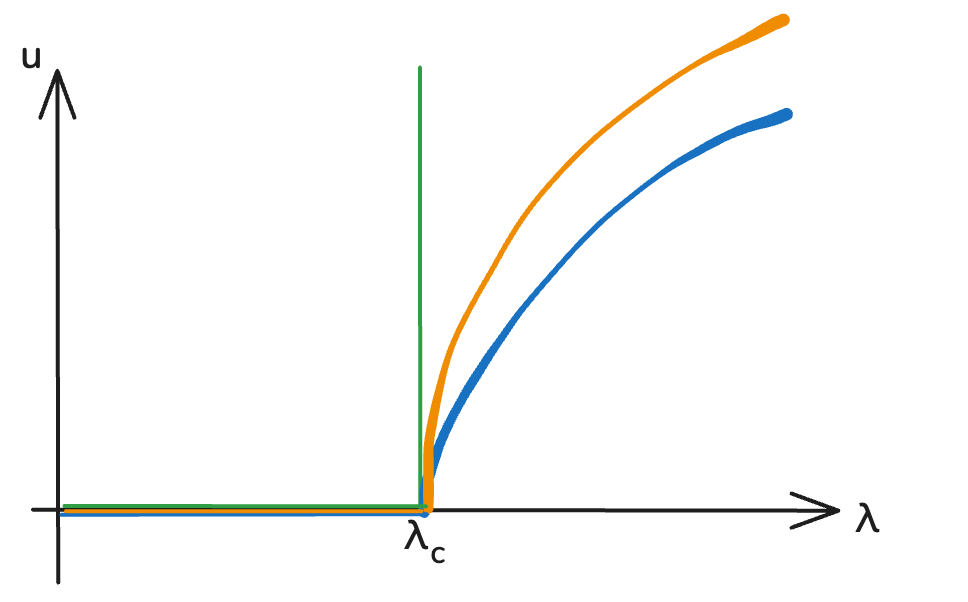}
  \caption{Scheme of the sets of solution of \eqref{kr} for various values of $\epsilon$, as functions of $\lambda$. The vertical axes can be interpreted as the $\|\cdot\|_\infty$ norm of the solution. In green, $\epsilon = 0$, in orange $\eps= \eps'>0$ and in blue $\eps > \eps'$.}
  \label{fig:1}
\end{figure}%

We believe that Figure \ref{fig:1} illustrates well the fact that the non-linearity, measured here with $\eps$ bends the straight continuum of solutions of the linear problem into the branch indexed by $\lambda \in [\lambda_c,1]$. Moreover, this idea is somehow justified by the next result.
\begin{Prop}\label{prop:3}
For any $\phi$ solution of \eqref{pb2} with $\eps = 0$ and $\lambda = \lambda_c$, there exists a function $\Lambda : (0,1]\to [\lambda_c, \infty)$ such that 
\[
\lim_{\eps \to 0}\|\phi - u_{\eps,\Lambda(\eps)}\|_\infty = 0.
\]
\end{Prop}
The proof of this result relies on the following elementary Lemma.
\begin{Lemma}
For $\eps > 0$, the function $\psi : \lambda \to \|u_{\eps,\lambda}\|_\infty$ is increasing on $[\lambda_c,\infty)$.
\end{Lemma}
\begin{proof}
Take $\eps > 0$. Thanks to Lemma \ref{lemma:uniq}, we know that $\psi$ is non-decreasing. Assume that for some $\lambda' > \lambda \geq \lambda_c$, we have $\psi(\lambda') = \psi(\lambda)$. This implies that there exists some $x^* \in \Omega$ such that for all $x \in \Omega, u_{\eps,\lambda}(x) \leq u_{\eps,\lambda'}(x) \leq u_{\eps,\lambda'}(x^*)= u_{\eps,\lambda}(x^*)$. Note that this also has to be true if we replace $\lambda'$ by $\lambda''$ for $\lambda'' \in [\lambda,\lambda']$. Evaluating the PDE in \eqref{pb2} at $x^*$ for these functions and using the fact that $x^*$ is a point of maximum of $u_{\eps,\lambda} - u_{\eps,\lambda'}$, we find out that
\[
u_{\eps,\lambda}(x^*) - u_{\eps,\lambda'}(x^*) \leq (\lambda u_{\eps,\lambda}(x^*) - \lambda' u_{\eps,\lambda'}(x^*))r(x^*).
\]
The previous leads to
\[
0 \leq \lambda( u_{\eps,\lambda}(x^*) -  u_{\eps,\lambda'}(x^*))r(x^*) + r(x^*)u_{\eps,\lambda'}(x^*)(\lambda - \lambda').
\]
Since $u_{\eps,\lambda'}(x^*)= u_{\eps,\lambda}(x^*)$, this implies that $r(x^*) = 0$ since $\lambda < \lambda'$, which is not possible because, evaluating the PDE satisfied by $u_{\eps,\lambda'}$ at $x^*$, this would lead to the fact that $u_{\eps,\lambda'}(x^*) = 0$ which would imply $u_{\eps,\lambda'} \equiv 0$ which is a contradiction. Hence, $\psi$ is increasing.
\end{proof}
\begin{proof}[Proof of Proposition \ref{prop:3}]
Let $\phi$ be a solution of \eqref{pb2} with $\eps = 0$ and $\lambda = \lambda_c$. For $\eps > 0$, consider $\Lambda(\eps)$ to be the unique $\lambda$ such that $\|u_{\eps,\lambda}\|_\infty = \|\phi\|_\infty$. Note that $\Lambda(\eps)$ is always well defined for $\eps$ small enough thanks to Proposition \ref{prop:mono}. Arguing as in the proof of Lemma \ref{lemma:apriori}, we deduce that $(\|\nabla_xu_{\eps,\Lambda(\eps)}\|_{L^2})_{\eps > 0}$ is a bounded sequence. Hence, from the PDE satisfied by $u_{\eps,\Lambda(\eps)}$, we obtain additional regularity bounds and thus the fact that $(u_{\eps,\Lambda(\eps)})_{\eps}$ converges, up to a subsequence, uniformly toward some limit. Because this limit has to satisfy the linear PDE, and because for all $\eps$, $\|u_{\eps,\Lambda(\eps)}\|_\infty = \|\phi\|_\infty$, we indeed obtain that $(u_{\eps,\Lambda(\eps)})_{\infty}$ converges uniformly toward $\phi$.
\end{proof}

Furthermore, we can establish that, for any $\eps > 0$,  the branch of solutions $(u_{\eps,\lambda})_{\lambda}$ is $C^\infty$.
\begin{Prop}
Take $\eps > 0$ and $\lambda > \lambda_c$. There exists $\delta > 0$ such that $F:(\lambda-\delta,\lambda+\delta) \to L^\infty(\Omega), \lambda \to u_{\eps,\lambda}$ is $C^\infty$.
\end{Prop}
\begin{proof}
We start by showing that $F$ is differentiable at $\lambda$. Consider $v= F(\lambda + h)$ for $h > \lambda_c- \lambda$. Denote $u = F(\lambda)$. Define $w = u -v$ and observe that
\[
-\nu\Delta w + \eps|\nabla_x u|^2 - \eps |\nabla_x v|^2 + w = -hr(x)u + (\lambda + h)r(x) w \text{ in } \Omega
\]
Take $h < 0$. We already know that $w >>0$ and $\|w\|_\infty\to 0$ as $h \to 0$. Define $g_h = -h^{-1}w$ and observe that, using the convexity of $x \to |x|^2$,
\[
-\nu \Delta g_h + \nabla_xv \cdot\nabla_x g_h + g_h - (\lambda +h) r(x)g_h \leq r(x)u \text{ in } \Omega.
\]
This inequality, as well as the obvious $\partial_n g_h = 0$ on $\partial \Omega$, yields a uniform bound on $\|g_h\|_\infty$. Indeed, the right hand side is constant, thus bounded when $h \to 0$ and the elliptic operator on the left hand side satisfies the maximum principle. The latter can be seen in the optimal control interpretation for instance, since its first eigenvalue is positive. This implies for instance that $g_h$ is bounded from above by the unique solution $\tilde v$ of
\[
\ba
-\nu \Delta \tilde v + |\nabla_x \tilde v|^2 + \tilde v - (\lambda +h) r(x)\tilde v = r(x) u \text{ in } \Omega,\\
\partial_n \tilde v = 0 \text{ on } \partial \Omega.
\ea
\]
Furthermore, thanks to the stability of this equation, see Proposition \ref{prop:stab}, we know that $\tilde v$ converges uniformly toward a bounded function when $h \to 0$.\\

Thus, because $\|g_h\|_\infty$ is bounded as $h \to 0$, we can indeed consider a weak limit $\tilde g$ of $g_h$ as $h \to 0$. It is clearly the weak solution of
\[
\ba
-\nu \Delta \tilde g + \nabla_xu \cdot\nabla_x \tilde g + \tilde g - \lambda r(x)\tilde g = r(x)u \text{ in } \Omega,\\
\partial_n \tilde g = 0 \text{ on } \partial \Omega.
\ea
\]
Such a solution is unique and the argument is once again that the elliptic operator on the left hand side satisfies the maximum principle. Furthermore, we obtain that $\tilde g$ is $C^\infty$. Taking $h>0$ would have lead to the same thing. Hence, we showed that $\tilde g = F'(\lambda)$.

The proof can now be concluded easily thanks to the implicit function theorem, and we leave the details to the interested reader.
\end{proof}

\subsection{Convergence of the time dependent problem}
In this section we are interested in the behaviour, when $T \to \infty$ of $u(T,\cdot)$ for $u$ the solution of 
\be\label{cauchy}
\ba
\partial_t u - \nu \Delta u + \frac 12 |\nabla_x u|^2 = a(x) u + f \text{ in } (0,\infty)\times \Omega,\\
u|_{t = 0} = u_0,\\
\partial_n u = 0 \text{ on } (0,\infty)\times \partial \Omega,
\ea
\ee
where $u_0,f$ are smooth non-negative function on $\Omega$. We have the following.
\begin{Theorem}\label{thm:cauchy}
There exists a unique smooth non-negative solution $u$ of \eqref{cauchy}. If $f \ne 0$ or $u_0 \ne 0$, then for all $t > 0,x \in \Omega$, $u(t,x) > 0$ and, if the assumptions of Theorem \ref{thm:1} are satisfied, $\lim_{T \to \infty}\|u(T,\cdot) -\bar u\|_\infty = 0$ where $\bar u$ is the solution of \eqref{krf} (or \eqref{kr} in the case $f = 0$) given by Theorem \ref{thm:2} (or Theorem \ref{thm:1} in the case $f = 0$).
\end{Theorem}
\begin{proof}
The existence of such of function can be obtained through the optimal control formula
\[
u(t,x) := \inf_{(\alpha_s)_{s \leq t}}\mathbb{E}\left[ \int_0^te^{\int_0^sa(X^{x,\alpha}_{t-u})du}\left(\frac 12 |\alpha_{t-s}|^2 + f(X^{x,\alpha}_{t-s})\right)ds + e^{\int_0^ta(X^{x,\alpha}_{t-s})ds}u_0(X^{x,\alpha}_t)\right],
\]
where $(X^{x,\alpha}_s)_{s \geq 0}$ is the process defined in \eqref{sde}. Since the previous formula also yields that $u$ is smooth (because so are $a$, $f$ and $u_0$), uniqueness comes from the regularity of the solution, for instance by using standard Gronwall's estimates on the difference of two solutions.

Using the control $\alpha_s = -\nabla_x \bar u(X^{x,\alpha}_{t-s})$ in the definition of $u$, we obtain thanks to Lemma \ref{lemma:41} that, for all $x \in \Omega$, $\limsup_{T \to \infty} u(T,x) \leq \bar u(x)$ as 
\be\label{eq:++}
\lim_{T \to \infty}\mathbb{E}\left[e^{\int_0^ta(X^{x,\alpha}_{t-s})ds}u_0(X^{x,\alpha}_t) \right] = 0.
\ee
 Approximating $\bar u$ with optimal control like formulas as we did in section \ref{sec:optcontr}, we arrive at the relation $\liminf_{T \to \infty}u(T,x) \geq \bar u (x)$ since $u_0 \geq 0$.
 
 It now remains to show that the convergence is uniform. Since \eqref{eq:++} can be made uniform in $x$, we deduce that $u$ is uniformly bounded, hence that it is also the case for its spatial gradient. Hence the result follows from the application of Ascoli-Arzela Theorem on any sequence of the form $(u(t_n,\cdot))_{n \geq 0}$ for $(t_n)_{n\geq 0}$ an increasing  sequence which tends to $+\infty$.

\end{proof}

\subsection{Extension to more general equations}
We briefly comment on some extensions of the results above. First, working with a purely quadratic Hamiltonian is not needed and we could have studied equations of the form
\be\label{krH}
\ba
-\nu\Delta u + H(u,\nabla_x u) = a(x)u \text{ in } \Omega,\\
u(x) > 0 \text{ in } \Omega,\\
\partial_n u = 0 \text{ on } \partial \Omega,
\ea
\ee
where $H: \R \times \R^d \to \R_+$ is an Hamiltonian upon which several assumptions can be made depending on the extension considered.\\

For the comparison principle in Lemma \ref{lemma:uniq}, we only used two properties on the Hamiltonian: i) that it is locally Lipschitz and ii) that it satisfies for any $v>0, p \in \R^d,\theta \in (0,1)$, $H(\theta v,\theta p) < \theta H(v,p)$. And existence can be studied even without this assumption! 

In the a priori estimate in Lemma \ref{lemma:apriori}, we only used the $2$-homogeneity of the Hamiltonian and its non-negativity. Note that any $p$-homogeneity with $p> 1$ is sufficient. Moreover, $L^\infty$ estimates could have been obtained without exact homogeneity assumptions, but simply with general inequality as the one we mentioned above. Also, techniques stemming from the existence of explosive super-solution, such as in Lasry and Lions \citep{lasry1}, could also have been used to obtain such an $L^\infty$ estimate. 

The same type of inequalities are sufficient in Lemma \ref{lemma:sub} and no particular additional assumption is required in the proof of Theorem \ref{thm:1}.\\

Furthermore, in the case $b \ne 0$, i.e. when the left hand side of the PDE of \eqref{kr} contains a term of the form $-b\cdot \nabla_x u$, then all the previous statements remains correct but the eigenvalue that we need to consider is obviously now $\lambda_1(-\nu \Delta - b\cdot \nabla_x -a)$.\\

Finally, under monotonicity assumptions such as in Proposition \ref{prop:5}, we could consider certain systems of coupled equations of the type \eqref{krH}.

\section{Economic consequences and interpretation of the mathematical analysis}\label{sec:eco}
We now specify particular instances of the class of models developed in section \ref{sec:derivation} and analyze the consequences of the results we just proved.

\subsection{Cryptocurrency bubbles as reserve goods against government instabilities}
We consider the case in which the numeraire is the fiat money (we assume that there is only one government) and the other good is a cryptocurrency. The state of the world is described by a variable $x$ that we take to be in $\Omega = [0,1]$. We assume that this variable represents the results of the monetary policies of the government. When $x$ is close to $0$, the taxes are low and inflation is also low. When $x$ is close to $1$, taxes are high and the inflation is also high. Here we shall assume that the state of the economy at time $t$ is given by $X_t$ where $(X_t)_{t \geq 0}$ is the strong solution of the reflected SDE
\[
dX_t = \sqrt{2 \nu}dW_t
\]
on $(0,1)$, for $\nu > 0$. For the sake of shortness, we shall call cash the fiat money and cryptos the cryptocurrencies.\\

In this model, the agents have the possibility to protect themselves from potential future bad monetary policies by buying cryptos, which they can then resell for cash at future times if they want to benefit from good monetary policies. Note that in order for such strategies to be useful, a bubble needs to exist on the price of the cryptos, as they do not produce any value in terms of cash. Hence their intrinsic value is $0$ and they are only bought if the agent speculate that it is going to be valued by other agents at later times, which is the definition of a speculative bubble. 

A natural consequence of the way we define this variable $x$ is to take the return $r_0$ of the government money as a decreasing function of $x$. We make the assumptions that $r_0$ is positive close to $x=0$ and that it is negative close to $x=1$.

On the other hand, the return $r_1$ of the cryptocurrency is negative and a constant in $x$. Recall that this return is measured in the unit of the cryptocurrency itself, so it not depending on $x$ is quite natural. The fact that it is negative arises from the actual implementation of some Blockchain, such as the Bitcoin Blockchain. Indeed, the total amount of Bitcoin is predetermined and some of it has to be spent to pay miners to continue ensuring sufficient security on the Blockchain. Finally we note $K$ the total number of Bitcoins and $N$ the number of agents in the economy. We focus here on the CARA case, denote the CARA parameter by $c$ and thus we note $\eps = c2K\nu/N$.\\

The PDE solved by the value of cryptos $u$ is thus given by
\be\label{bitcoin}
\ba
-\nu u'' + \eps (u')^2 = (r_1 - r_0)(x)u \text{ in } (0,1),\\
u' = 0 \text{ at } 0 \text{ and } 1,\\
u \geq 0.
\ea
\ee

We shall say that a bubble consensus value (C-value) on crypto exists if there is a non-negative non-zero solution of \eqref{bitcoin}, and the solution $u$ is then called the C-value.\\

\subsection{Economic conclusions in the cryptocurrency case}
We now translate the mathematical results of Section \ref{sec:math} in this model.\\

\textit{1) If there is a possibility the government does bad enough and one that it outperforms cryptos in terms of relative returns, then there is a unique C-value on crypto. It is thus the unique equilibrium price of this bubble which is a function of the state of the economy. Taking the effect of the government as variable, there is an explicit threshold after which the C-value exists. It can be computed as the first eigenvalue of an elliptic operator. The C-value is stable (i.e. evolves continuously) with respect to all the parameters of the model, except when $\eps \to 0$, see section \ref{sec:arbitrage}.}\\

This fact follows simply from Theorem \ref{thm:1}. The condition on $\lambda_1(-\nu \partial^2_x +(r_0-r_1))$ will be satisfied if $r_0$ is sufficiently negative near $x=1$, which amounts to say that close to $x=1$ (i.e. the worst case scenario for cash), inflation and taxes can be huge enough. The assumption that $(r_0-r_1)$ should change sign is verified if $r_0(0) > r_1$. The stability is a consequence of Proposition \ref{prop:stab}.\\

\textit{2) If a C-value exists, it increases as the economic conditions $X_t$ deteriorates.}\\

This is simply the consequence of Proposition \ref{prop:5}.\\

\textit{3) If a C-value exists, it increases with the number of agents willing to buy it.}\\

In our model, the number of agents is denoted by $N$. Hence, as $N$ grows, $\eps$ decreases. Hence, as a consequence of Proposition \ref{prop:mono}, the C-value increases (everywhere in $x$) as $N$ grows.\\

\textit{4) The C-value collapses to $0$ if the return of cash increases above a certain threshold, or if the one of crypto decreases below a certain level, or if there are no users, but not if the number of users decreases below a certain positive threshold. Reversely, even if a very small number of agents are interested in cryptos, it may be rational for the price of cryptos not to be zero.}\\

As we saw in the previous facts, as $N$ decreases, the parameter $\eps$ increases, which thus makes the value of the crypto decrease. However, this value $u$ keeps satisfying $u>>0$ for $N> 0$ arbitrary small and only reaches $0$ when $N = 0$. However, if $r_0$ increases sufficiently so that $\lambda_1(-\nu \partial_x^2 + (r_0-r_1)) = 0$, then instantly $u\equiv 0$, even if $r_0(1)$ is still negative. In particular, a bubble can exist with very few agents which are interested or concerned. The bubble will then be quite small as the price $u$ will be low, but it will still exhibits the stability properties proved in section \ref{sec:math}.\\

A reformulation of the previous fact can be.\\

\textit{4 bis) Even when very few persons were concerned about Bitcoin, it was rational that a very small but non-zero C-value existed at the time.}\\

\textit{5) The C-value increases as the return of cryptos increases. It also increases as the return of fiat money decreases.}\\

This is simply a consequence of Proposition \ref{prop:mono2}.\\

\textit{6) The C-value has no particular monotonicity with respect to the parameter $\nu$.}\\

This reminds us that $\nu$ does not play the role of risk but rather of volatility, as risk is a consequence of both volatility and the variation of $r_0$ in $x$. Nonetheless, if $\nu$ goes to $\infty$, the C-value goes to $0$. This is easily seen as, dividing \eqref{bitcoin} by $\nu$, we obtain that this corresponds to the limit $r_0,r_1 \to 0$.

\subsection{The case of gold}
From the point of view of our model, if we replace cryptocurrencies with gold, we end up with a fairly similar situation so we do not detail it here. 

\subsection{The model of Biais, Rochet and Villeneuve and the problem of the governments}
In \citep{brv}, the authors proposed a model to justify the use of cryptocurrencies by agents to protect themselves against non-benevolent governments which could increase taxes and inflation. In their model, the randomness is apparent in the evolution of the wealth on the individuals agents as well as on the possibility for the crypto asset to crash. We took several ideas from their modelling even though we ended up with quite a different problem. To transpose their models in terms of ours, we could say that in their model, the government chooses the level of the variable $X_t$, which has no intrinsic variation such as the one given by the SDE in our model. Idiosyncratic and exogenous variations of the wealth of the agents could have been considered when deriving our optimal investment $\theta^*$, and the probability for the price of the "other" good to go to $0$ instantly could also been considered. Note that the last phenomenon simply consists, given a constant intensity of crash $\beta$, to replace $r_1$ by $r_1 - \beta$.\\

In our framework, the best way to incorporate decisions from the government should be to assume that the government controls the SDE which drives the evolution of $(X_t)_{t \geq 0}$. For instance, that the government  could choose an adapted process $(\alpha_t)_{t \geq 0}$ and $(X_t)_{t \geq 0}$ would be given as the solution of
\[
dX_t = \alpha_t dt + \sigma dW_t,
\]
omitting the terms from the reflection to simplify the discussion. By committing to setting $\alpha_t = b(X_t)$ for a function $b$ that the government chooses and enforces, the problem of the latter is the one of the optimal control of the price $u$, solution of \eqref{pdev1} for instance, through the choice of $b$. This would very much be in the flavor of the work \citep{brv}. Remark that in such cases, by choosing the function $b$, the government in fact controls the $\lambda_1(-\nu \partial_x^2 - b \partial_x + (r_0-r_1))$. Hence, it can increase it so that the only bubble equilibrium price is $0$.\footnote{A simple computation yields that this is always possible for the government.}

Another, probably more realistic approach, could consists in saying that the government does not commit to anything, and that at time $t$, it simply chooses the $\alpha_t$ that suits it best. In this framework, the government now faces a standard stochastic optimal control problem, but it is coupled with the problem of all agents. In the limit of an infinite number of agents, we then end up with a mean field game with a major player, which is known to raise serious mathematical difficulties and of whose well-posedness is not understood at the moment and only partial results are known \cite{edmond,carmona2018probabilistic,lasrylionsmajor,bertucci2020strategic}. Even with a finite number of agents, we end up with a general $N+1$ players differential game, and such games are known to be notoriously difficult to analyze.

\subsection{Links with a no-arbitrage equation}\label{sec:arbitrage}
We explain here why the case $\eps = 0$ in \eqref{bitcoin} is the no-arbitrage equation in such a model.

If there is no arbitrage, then we must find that for any time length $h > 0$, buying an amount $q$ of cryptos for a total value of $1$ at time $t$ and selling it at time $t+h$ yields no revenue compared to holding the amount $1$ in cash between $t$ and $t+dt$. The latter yields a revenue of 
\[
\mathbb{E}\left[e^{\int_t^{t +h}r_0(X_s)ds}|X_t\right] -1,
\]
while the former yields
\[
\frac{\mathbb{E}\left[e^{\int_t^{t+h}r_1(X_s)ds}u(X_{t+h})|X_t\right]}{u(X_t)} -1.
\]
Dividing the two equations by $h$ and taking the limit $h \to 0$ yields $r_0(X_t)$ in the first expression and 
\[
r_1(X_t) + \frac{\sigma^2 u''(X_t)}{2u(X_t)}
\]
in the second one. Equalling the two quantities and multiplying by $u(X_t)$ we find precisely \eqref{bitcoin} in the case $\eps = 0$. This equation has only $0$ as a solution in general, i.e. when $0$ is not an eigenvalue of $-\sigma^2/2\partial^2_x + (r_0 - r_1)$.\\

Hence, in our model, risk neutral arbitragers should make the price of cryptos goes to $0$, while the risk aversion of our agents create a non-zero price. Indeed, in \eqref{bitcoin}, the parameter $\eps > 0$ is a consequence of the risk aversion of our agents. Note that the price of cryptos does not behave in a continuous manner in the limit $\eps \to 0$, as it explodes whereas, at the limit, it is simply $0$.

We believe this is a simple illustration of the classical fact that risk aversion prevents in general arbitrages to be made.

\subsection{The case of real estate in certain locations}
We now want to explain why in certain cases, real estate can play the same role as cryptos or gold for the protection from potential bad public policies.

Of course real estate is not by essence limited in quantity, nor it is without any return. Nonetheless, in certain precise locations such as the centre of some major cities, we argue that real estate indeed behaves as a bubble asset. First, it is clear that if one add the location of the real estate, then we can freely assume that it exists in finite and fixed quantity only, since such places typically do not allow for new buildings. Second, even if real estate produces returns in the form of rent, the market price to buy real estate seems arbitrary large compared to expected returns in rent that it can generates. A good example of such a situation is Paris, where a firm limitation on rent has been instated, and it has not translated in any sort of limit for the market price of real estate, although strong french protection of the renter should translates in a rational decision to rent its main house rather than buying it. We argue, quite classically, that it is because a non-negligeable proportion of its price behaves as the price of a bubble asset, and that the observed stability of this price can be explained by our mathematical analysis.\\

Assuming in this case that real estate is the "other" good, and that it produces a revenue $f \geq 0 $ while facing, like cryptos, a constant negative interest rate because of usual cost of exploitation, we end up with the following equation to characterize the price of real estate
\be\label{realestate}
\ba
-\nu u'' + \eps (u')^2 = (r_1 - r_0)(x)u + f\text{ in } (0,1),\\
u' = 0 \text{ at } 0 \text{ and } 1,\\
u \geq 0,
\ea
\ee
where $\eps$ is here given, in the case of the CRRA utility for instance, as $\eps = \gamma K\sigma^2/Q$, where $\gamma$ is the risk paramter, $K$ the volume of available real estate, $\sigma$ the volatility parameter and $Q$ the total wealth. We keep $r_0$ as a non-decreasing function. In this case, under certain assumptions, the price of real estate may include a bubble C-value component.

\subsection{Stylized facts in the case of real estate}
\textit{7) Even with a strong limit on the rent, the price of real estate can still be rationally arbitrary large in certain prestigious locations, i.e. it can contain a C-value component which can be arbitrary large compared to $f$.}\\

Thanks to Theorem \ref{thm:2}, we know that if $r_0$ is such that $\lambda_1(-\nu \partial^2_x +(r_0-r_1)) < 0$, then the unique solution of \eqref{realestate} is larger than the solution associated to the same problem with $f = 0$. The latter goes to $+\infty$ as $\eps \to 0$, thus independently of $f$, hence the result.\\

\textit{8) The more unique the place, the higher its C-value component.}\\

This can be seen once again through the monotone dependency of $u$ in the parameter $\eps$. Indeed, the rarity of the place translates as the fact that the total number of comparable places $K$ decreases. Hence the result.

\subsection{Have bubbles disappeared ?}
Not entirely. As we mentioned in section \ref{sec:derivation}, we are only able to characterize through our PDE prices which are of the form $p_t = u(X_t)$. Hence, the uniqueness of solutions to \eqref{bitcoin} for instance, does not prevent a priori the existence of more bubbly equilibrium paths, where $p_t$ just grows arbitrary large. We do not believe this to be an issue as this type of phenomenon could occur on any good.\\

Another interesting way to see arbitrary large bubbles appear is through the equation \eqref{pdev3} obtained in the CRRA case. In the case where almost all wealth $Q_t$ is invested into the "other" good, it is natural to assume $Q[u](x)$ is some linear function of $u$. In fact, it could almost be $K u(x)$. In such a situation, even if a non-linearity is present in \eqref{pdev3}, because it is homogenous of order $1$ in $u$, our mathematical analysis completely fails as the equation becomes homogenous of order $1$, and thus, if it has a non-zero solution, then it necessary have a infinite number of them. In particular, it has an arbitrary large solutions, which we believe is an interesting way of understanding how unstable bubbles can be created from a stable one.

\subsection{What constitues a potential bubble asset ?}
Following Samuelson \citep{samuelson1958} and Wallace \citep{wallace}, several characteristics of money seem to be of importance: exchangeability, non-perishability and uselessness. In our model, we did not focus on exchangeability and assume that it is always the case. However we provided insights on the other two properties and we now take the time to explain why.

The perishability of money can often be thought of in terms of inflation or interest rates. In our model, we have proven that for a C-value to form on an asset, it needs to be "less" perishable than fiat money, with a quantitative estimate of this non-perishability which is quite non-obvious since it is measured through the sign of the first eigenvalue of an elliptic operator. Hence, a low perishability is important, and a small one is tolerated by the agent to form a C-value.\\

Our model and the analysis above does not need the bubble asset to be useless, which can seem quite counterintuitive for such assets. In our model, agents buy the bubble asset as an insurance, to protect themselves from inflation for instance. So it is natural they have no problem with their asset being useful. Because we are not restricting ourselves to bubble on money assets, this uselessness condition does not appear here. Then the questions remains to understand why C-value seems to form in a stronger manner on useless assets in practice. We believe that this may be due to the fact that non-perishability and uselessness might be linked. Indeed, if an agent owns an extremely useful good, it is possible that this agent anticipates the risk of this good to be taken from him because of its usefulness, thus decreasing its non-perishability from its own point of view. 

\section{Conclusion and future perspectives}
We proposed a radically novel way to justify the stability of the price of certain bubble assets. Our method relies on the mathematical analysis of a simple yet new PDE. We have exhibited conditions under which it is rational for the agents to form a consensus on a non zero price for assets which generates nothing, as it allows them to protect themselves against public policies which are potentially armful to their portfolio in government money, or more generally in traditional assets. This bubble equilibrium price enjoys very robust stability and uniqueness properties, and furthermore it can be characterized as the unique positive solution of a quite simple but non-linear PDE.\\

Even though our story is not complete yet, we believe that this approach could be the entry point for several studies. We indicate here three directions of investigation which we think should bear some interest and upon which we are currently working.
\begin{itemize}
\item The derivation of the PDE at interest in the CRRA case could be addressed with more care, in particular with regards to the variation of the total wealth.
\item More generally, the demand of "other" good could depend on the entire repartition of wealth in the population, as in real life, borrowing constraints usually prevents agents with low wealth to invest as much as they would like. This extension should lead to interesting modelings much more in the spirit of mean field games \citep{lasry2007mean}. The same goes for considering a population of heterogenous agents, in their risk parameters for instance.
\item We have explained that several goods can serve as protection. The natural following question is to understand why certain goods should be favored compared to others. This type of questions leads to the study of coupled equations of the form of \eqref{bitcoin}, for which no mathematical analysis exists in the literature at the moment, and that we are currently investigating.
\end{itemize}

\bibliographystyle{plainnat}
\bibliography{bibbubble}

\begin{thebibliography}{15}
\providecommand{\natexlab}[1]{#1}
\providecommand{\url}[1]{\texttt{#1}}
\expandafter\ifx\csname urlstyle\endcsname\relax
  \providecommand{\doi}[1]{doi: #1}\else
  \providecommand{\doi}{doi: \begingroup \urlstyle{rm}\Url}\fi

\bibitem[Achdou et~al.(2022)Achdou, Bertucci, Lasry, Lions, Rostand, and
  Scheinkman]{edmond}
Yves Achdou, Charles Bertucci, Jean-Michel Lasry, Pierre-Louis Lions, Antoine
  Rostand, and Jos{\'e}~A Scheinkman.
\newblock A class of short-term models for the oil industry that accounts for
  speculative oil storage.
\newblock \emph{Finance and Stochastics}, 26\penalty0 (3):\penalty0 631--669,
  2022.

\bibitem[Bertucci et~al.(2020)Bertucci, Lasry, and
  Lions]{bertucci2020strategic}
Charles Bertucci, Jean-Michel Lasry, and Pierre-Louis Lions.
\newblock Strategic advantages in mean field games with a major player.
\newblock \emph{Comptes Rendus. Math{\'e}matique}, 358\penalty0 (2):\penalty0
  113--118, 2020.

\bibitem[Biais et~al.(2024)Biais, Rochet, and Villeneuve]{brv}
Bruno Biais, Jean-Charles Rochet, and St\'ephane Villeneuve.
\newblock Do cryptocurrencies matter?
\newblock \emph{Working paper}, 2024.

\bibitem[Carmona et~al.(2018)Carmona, Delarue,
  et~al.]{carmona2018probabilistic}
Ren{\'e} Carmona, Fran{\c{c}}ois Delarue, et~al.
\newblock \emph{Probabilistic Theory of Mean Field Games with Applications
  I-II}.
\newblock Springer, 2018.

\bibitem[Fleming and Soner(2006)]{fleming}
Wendell~H Fleming and Halil~Mete Soner.
\newblock \emph{Controlled Markov processes and viscosity solutions},
  volume~25.
\newblock Springer Science \& Business Media, 2006.

\bibitem[Laetsch(1975)]{laetsch1975uniqueness}
Theodore Laetsch.
\newblock A uniqueness theorem for elliptic quasi-variational inequalities.
\newblock \emph{J. Functional Analysis}, 18:\penalty0 286--287, 1975.

\bibitem[Lasry and Lions(1989)]{lasry1}
Jean-Michel Lasry and Pierre-Louis Lions.
\newblock Nonlinear elliptic equations with singular boundary conditions and
  stochastic control with state constraints: 1. the model problem.
\newblock \emph{Mathematische Annalen}, 283:\penalty0 583--630, 1989.

\bibitem[Lasry and Lions(2007)]{lasry2007mean}
Jean-Michel Lasry and Pierre-Louis Lions.
\newblock Mean field games.
\newblock \emph{Japanese Journal of Mathematics}, 2\penalty0 (1):\penalty0
  229--260, 2007.

\bibitem[Lasry and Lions(2018)]{lasrylionsmajor}
Jean-Michel Lasry and Pierre-Louis Lions.
\newblock Mean-field games with a major player.
\newblock \emph{Comptes Rendus Mathematique}, 356\penalty0 (8):\penalty0
  886--890, 2018.

\bibitem[Lions(2020-2021)]{lions2021cours}
Pierre-Louis Lions.
\newblock Cours au college de france.
\newblock \emph{Available at www. college-de-france. fr}, 2020-2021.

\bibitem[Lions and Sznitman(1984)]{lionssznitman}
Pierre-Louis Lions and Alain-Sol Sznitman.
\newblock Stochastic differential equations with reflecting boundary
  conditions.
\newblock \emph{Communications on pure and applied Mathematics}, 37\penalty0
  (4):\penalty0 511--537, 1984.

\bibitem[Merton(1969)]{merton}
Robert~C Merton.
\newblock Lifetime portfolio selection under uncertainty: The continuous-time
  case.
\newblock \emph{The review of Economics and Statistics}, pages 247--257, 1969.

\bibitem[Samuelson(1958)]{samuelson1958}
Paul~A Samuelson.
\newblock An exact consumption-loan model of interest with or without the
  social contrivance of money.
\newblock \emph{Journal of political economy}, 66\penalty0 (6):\penalty0
  467--482, 1958.

\bibitem[Tirole(1985)]{tirole}
Jean Tirole.
\newblock Asset bubbles and overlapping generations.
\newblock \emph{Econometrica: Journal of the Econometric Society}, pages
  1499--1528, 1985.

\bibitem[Wallace(1978)]{wallace}
Neil Wallace.
\newblock The overlapping-generations model of fiat money.
\newblock \emph{Staff Report}, 1978.

\end{thebibliography}

\end{document}